\documentclass[11pt]{amsart}

\usepackage{amsmath}
\usepackage{amsfonts}
\usepackage{amssymb}
\usepackage{amscd}
\usepackage{amsthm}
\usepackage{framed}
\usepackage{fullpage}
\usepackage{graphicx}
\usepackage{latexsym}
\usepackage[numbers]{natbib}

\usepackage{hyperref}
\hypersetup{colorlinks,linkcolor={red},citecolor={blue},urlcolor={blue}}

\newtheorem{theorem}{Theorem}[section]
\newtheorem{lemma}[theorem]{Lemma}
\newtheorem{proposition}[theorem]{Proposition}
\newtheorem{corollary}[theorem]{Corollary}
\newtheorem{conjecture}[theorem]{Conjecture}

\theoremstyle{definition}

\newtheorem{remark}[theorem]{Remark}

\numberwithin{equation}{section}

\newcommand{\CC}{\mathbb C}
\newcommand{\HH}{\mathbb H}

\newcommand{\QQ}{\mathbb Q}

\newcommand{\ZZ}{\mathbb Z}

\newcommand{\Mp}{\mathop{\mathrm {Mp}}\nolimits}
\newcommand{\SL}{\mathop{\mathrm {SL}}\nolimits}

\newcommand{\Sp}{\mathop{\mathrm {Sp}}\nolimits}
\newcommand{\Orth}{\mathop{\null\mathrm {O}}\nolimits}

\newcommand{\rank}{\mathop{\mathrm {rk}}\nolimits}

\def\Grit{\mathbf{G}}
\def\Borch{\mathbf{B}}
\def\m{\operatorname{mod}}

\def\det{\operatorname{det}}
\def\w{\operatorname{w}}

\newenvironment{psmallmatrix}
  {\left(\begin{smallmatrix}}
{\end{smallmatrix}\right)}

\begin{document}

\title[Mathieu moonshine and Borcherds products]{Mathieu moonshine and Borcherds products}

\author{Haowu Wang}

\address{Center for Geometry and Physics, Institute for Basic Science (IBS), Pohang 37673, Korea}

\email{haowu.wangmath@gmail.com}

\author{Brandon Williams}

\address{Lehrstuhl A für Mathematik, RWTH Aachen, 52056 Aachen, Germany}

\email{brandon.williams@matha.rwth-aachen.de}

\subjclass[2020]{11F50, 11F46, 81T30}

\date{\today}

\keywords{Borcherds products, Jacobi forms, Mathieu group $M_{24}$,Twisted genera of $K3$ surfaces}

\begin{abstract} 
The twisted elliptic genera of a $K3$ surface associated with the conjugacy classes of the Mathieu group $M_{24}$ are known to be weak Jacobi forms of weight $0$. In 2010, Cheng constructed formal infinite products from the twisted elliptic genera and conjectured that they define Siegel modular forms of degree two. In this paper we prove that for each conjugacy class of level $N_g$ the associated product is a meromorphic Borcherds product on the lattice $U(N_g)\oplus U \oplus A_1$ in a strict sense. We also compute the divisors of these products and determine for which conjugacy classes the product can be realized as an additive (generalized Saito--Kurokawa) lift. 
\end{abstract}

\maketitle

\section{introduction}

Monstrous moonshine, proposed by Conway--Norton \cite{CN79} and proved by Borcherds \cite{Bor92}, is the phenomenon that the Fourier coefficients of the modular $J$-function and the Hauptmoduls of other genus zero subgroups of $\mathrm{SL}_2(\mathbb{R})$ have simple expressions involving the dimensions of representations of the monster group. 

Eguchi, Ooguri and Tachikawa observed \cite{EOT11} that the Fourier coefficients of the elliptic genus of a $K3$ surface have similar expressions, now involving the dimensions of representations of the largest Mathieu group $M_{24}$. This phenomenon has been generalized to twists by all elements of $M_{24}$ and is called \emph{Mathieu moonshine}. Cheng, Duncan and Harvey \cite{CDH14a,CDH14b} further generalized this phenomenon to the umbral moonshine, which relates Niemeier lattices and Ramanujan's mock theta functions.

The elliptic genus of a $K3$ surface is a weak Jacobi form of weight $0$ and index $1$ on $\SL_2(\ZZ)$. By the work of Cheng \cite{Che10}, Gaberdiel--Hohenegger--Volpato \cite{GHV10a, GHV10b} and Eguchi--Hikami \cite{EH11}, one can associate to each conjugacy class of $M_{24}$ a twist of the elliptic genus which plays the role of the McKay--Thompson series in Monstrous moonshine. If $[g]$ is a conjugacy class of $M_{24}$ of order $n_g$ and level $N_g$, then the twisted elliptic genus corresponding to $[g]$ has the expression
\begin{equation}
\begin{split}
\phi_g(\tau,z)&= \frac{\chi(g)}{12}\phi_{0,1}(\tau, z) + \tilde{T}_g(\tau)\phi_{-2,1}(\tau,z)\\    
&=\sum_{n\geq 0}\sum_{r\in \ZZ} c_g(4n-r^2)q^n\zeta^r,
\end{split}
\end{equation}
where $\chi(g):=\phi_g(\tau,0)$ is the Witten index, the character of a certain permutation representation of $M_{24}$ on $24$ points; $\tilde{T}_g$ is a certain holomorphic modular form of weight $2$ and trivial character on $\Gamma_0(N_g)$; $\phi_{0,1}$ and $\phi_{-2,1}$ are the basic weak Jacobi forms introduced by Eichler--Zagier \cite{EZ85}; and $q=e^{2\pi i\tau}$ and $\zeta=e^{2\pi iz}$ for $(\tau, z) \in \HH\times\CC$. Altogether $\phi_g$ is a weak Jacobi form of weight $0$ and index $1$ on $\Gamma_0(N_g)$ with trivial character. The forms $\tilde{T}_g$ involve mock modular forms of weight $\frac{1}{2}$ appearing in Mathieu moonshine \cite{CD12b, CD12}.  Gannon \cite{Gan16} has proved that the Fourier coefficients of these mock modular forms are true characters of $M_{24}$. 

The Igusa cusp form $\Phi_{10}$ \cite{Igu62} is related to the $K3$ elliptic genus $2\phi_{0,1}$ through the Borcherds multiplicative lift \cite{Bor95, GN97, Bor98}. The inverse of $\Phi_{10}$ gives the second-quantized elliptic genus of a $K3$ surface \cite{DMVV97} and the partition function for $\frac{1}{4}$-BPS states of type II superstring theory compactified on $K3\times T^2$ \cite{DVV97}, and the square root of $\Phi_{10}$ is the denominator of a generalized Kac--Moody algebra \cite{GN97}.  Cheng \cite{Che10} observed that an action of $M_{24}$ on a certain algebraic object related to this algebra would lead to expressions for the twists by $g\in M_{24}$ as infinite products involving the $K3$ elliptic genera twisted by $g^k$ for $k|n_g$. Based on this observation, Cheng and Duncan \cite[\S 4]{CD12} associated to each conjugacy class $g$ of $M_{24}$ a formal infinite product 
\begin{equation}\label{eq:infprod}
 \Phi_g(\tau,z,\omega) = q\zeta s  \prod_{(n,r,m)>0}\exp\left( -\sum_{a=1}^\infty \frac{c_{g^a}(4nm-r^2)}{a} \Big(q^n\zeta^r s^m\Big)^a \right),
\end{equation}
where $s=e^{2\pi i\omega}$ and $(n,r,m)>0$ means that either $m>0$, or $m=0$ and $n>0$, or $m=n=0$ and $r<0$. They further made the following precise conjecture.

\begin{conjecture}[Conjecture 4.1 of \cite{CD12}]\label{conj}
For each conjugacy class $g$ of $M_{24}$, the product $\Phi_g$ defines a Siegel modular form with some multiplier on the congruence subgroup 
$$
\Gamma_0^{(2)}(N_g)=\left\{ \begin{psmallmatrix}
A & B \\ C & D
\end{psmallmatrix} \in  \Sp_4(\ZZ) : C=0 \m N_g \right\}.
$$
\end{conjecture}

This conjecture was also motivated by the connection via wall-crossing between the $\frac{1}{4}$-BPS spectrum and the $\frac{1}{2}$-BPS spectrum in string theory. The regularized value of $1/\Phi_g$ at $z=0$ is a product of two copies of the twisted $\frac{1}{2}$-BPS partition functions $1/\eta_g$, where $\eta_g$ is the eta product determined by the cycle shape of $g$. As notations in \cite[Table 1]{CD12}, for $g$ of type $2A, 3A, 5A, 7AB$ the products $\Phi_g$ have interpretations as partition functions for $\frac{1}{4}$-BPS states twisted by a $\ZZ/p$ symmetry for $p = 2, 3, 5, 7$ \cite{JS06, DJS07,DN07,Sen10a, Sen10b, GK10, Gov12}. As pointed out by Cheng \cite{Che10}, a positive solution to this conjecture would provide strong evidence for the existence of the underlying $M_{24}$ symmetry and a general theory of twisted $\frac{1}{4}$-BPS spectrum. A modular construction of $\Phi_g$ has also appeared in \cite{GK09, EH12} for some particular conjugacy classes.  

When $g$ is of type $1A$, $\Phi_g$ is exactly the Igusa cusp form $\Phi_{10}$. Aoki and Ibukiyama \cite{AI05} proposed a construction of Siegel modular forms on $\Gamma_0^{(2)}(N)$ as Borcherds products of Jacobi forms of weight $0$ and index $1$ on $\Gamma_0(N)$ under certain assumptions. This construction was given in general by Cl\'{e}ry and Gritsenko \cite{CG11} who also extended this construction to congruence subgroups of paramodular groups of degree two. In this way Cl\'{e}ry and Gritsenko constructed the square roots of $\Phi_g$ for $g$ of type $2A, 3A, 4B$ and proved that they and $\Phi_{10}$ are the only Siegel modular forms on groups of type $\Gamma_0^{(2)}(N)$ which vanish precisely on the diagonal. Raum \cite{Rau13} proved Conjecture \ref{conj} for all conjugacy classes whose level is a prime power; more precisely, Raum showed that these $\Phi_g$ are products of rescalings of Borcherds products in the sense of Cl\'{e}ry and Gritsenko. Persson and Volpat \cite{PV14} gave a full proof of Conjecture \ref{conj} by generalizing  $\Phi_g$ to an analogue $\Phi_{g,h}$ for each commuting pair $g,h\in M_{24}$, and Govindarajan and Samanta \cite{GS21} extended Raum's method of proof to all conjugacy classes. 

In this paper we give a different proof of the conjecture. One important point is that we prove $\Phi_g$ extends to a single-valued meromorphic function on all $\HH_2$, which does not seem to be clear a priori as the zeros and poles of $\Phi_g$ do not necessarily intersect the region where the infinite product converges. We represent $\Phi_g$ as the Borcherds lift of an explicit vector-valued modular form for the Weil representation of $\Mp_2(\mathbb{Z})$. This allows us to explicitly compute the divisor of $\Phi_g$ and to determine when $\Phi_g$ can be constructed as an additive lift. This was unexpected before (e.g. \cite{PV14}). 

Our approach is based on an isomorphism recently established in \cite{WW22}:

\begin{theorem}[Theorem 5.1 of \cite{WW22}]\label{th:iso}
Let $U$ be the unique even unimodular lattice of signature $(1,1)$ and $L$ be an even positive-definite lattice of rank $\rank(L)$. Let $N$ be a positive integer.
There is an isomorphism between the space of weakly holomorphic modular forms of weight $k-\frac{1}{2}\rank(L)$ for the Weil representation of $\Mp_2(\ZZ)$ attached to $U(N)\oplus L$ and the space of tuples of Jacobi forms
$$
\bigoplus_{d|N} J_{k,L}^!(\Gamma_1(N/d)),
$$
where $J_{k,L}^!(\Gamma_1(t))$ is the space of weakly holomorphic Jacobi forms of weight $k$ and index $L$ on $\Gamma_1(t)$. 
\end{theorem}

This isomorphism is used in \cite{WW22} to classify Borcherds products of singular weight and to prove the moonshine conjecture for the Conway group proposed by Borcherds \cite{Bor90, Bor92} and Scheithauer \cite{Sch04, Sch06}: all twists of the denominator identity of the fake monster Lie algebra by elements of Conway's group $\mathrm{Co}_0$ are the Fourier expansions of Borcherds products of singular weight on certain symmetric domains of type IV.

It is well known that $\Sp_4(\ZZ)$ can be realized as a subgroup of the orthogonal group $\Orth^+(2U\oplus A_1)$ and Siegel modular forms of degree two can be viewed as modular forms on orthogonal groups of signature $(3,2)$ (see e.g. \cite{GN98}). By applying Theorem \ref{th:iso} to the $A_1$ root lattice we identify the inputs into Borcherds' lift to $\Gamma_0^{(2)}(N_g)$ with certain sequences of weakly holomorphic Jacobi forms $(\phi_d)$, $d | N_g$, where each $\phi_d$ has level $\Gamma_0(N_g/d)$.

We prove the following form of Conjecture \ref{conj}:

\begin{theorem}\label{mth1}
Let $g$ be a conjugacy class of $M_{24}$ of cycle shape $\prod t^{b_t}$ and level $N_g$. Then $\Phi_g$ is a meromorphic Borcherds product of weight $k_g=\frac{1}{2}\sum_{t=1}^\infty b_t - 2$ on the lattice $U(N_g)\oplus U\oplus A_1$. The preimage of $\Phi_g$ under Borcherds' lift is the family $$\phi_{g^d}, \quad d | N_g.$$ The divisor of $\Phi_g$ is computed explicitly in Appendix A. In particular we find the following:
\begin{enumerate}
    \item $\Phi_g$ is holomorphic if and only if $N_g\leq 4$, i.e. if $g$ is of type $1A$, $2A$, $2B$, $3A$, $4B$; 
    \item $\Phi_g$ has a meromorphic square root on $\HH_2$ if and only if $g$ is of type $1A$, $2A$, $2B$, $3A$, $3B$, $4A$, $4B$, $4C$, $5A$, $6A$, $6B$, $10A$, $11A$.
\end{enumerate}
\end{theorem}

Since each $\phi_{g^d}$ has rational Fourier expansion at infinity, \cite{WW22} shows that the associated vector-valued modular form also has rational Fourier coefficients and therefore there exists a denominator $d_g \in \mathbb{N}$ such that $\Phi_g^{d_g}$ is a meromorphic Borcherds product. The nontrivial point of Theorem \ref{mth1} is that we can take $d_g = 1$. The proof involves calculating the Fourier expansions of all $\phi_g$ at all cusps. Note that the divisors of $\Phi_g$ can be rather complicated, (e.g. for $g$ of type $12B$ and $21AB$), but these zeros and poles have quite low multiplicities. 

It was asked by Cheng--Duncan \cite{CD12} and Eguchi--Hikami \cite{EH12} whether $\Phi_g$ can be constructed as an additive lift of a Jacobi form when $g$ is one of the $12$ cycle shapes for which $\Phi_g$ has positive weight. We find a complete answer to this question here. Since most $\Phi_g$ have poles, we have to interpret this as the singular additive lift due to Borcherds \cite[Theorem 14.3]{Bor98}. The natural candidate for the input into the additive lift is the leading Fourier--Jacobi coefficient of $\Phi_g$,
\begin{equation}
\varphi_g(\tau,z):=\eta_g(\tau) \phi_{-2,1}(\tau,z),    
\end{equation}
which is a weak Jacobi form of weight $k_g$ and index $1$ on $\Gamma_0(N_g)$ with a Dirichlet character \cite{CD12}. By comparing divisors we show that the additive lift of $\varphi_g$ is a Borcherds product, which is related to (but does not always equal) $\Phi_g$.

\begin{theorem}\label{mth2}
The product $\Phi_g$ can be constructed as a Borcherds additive lift if and only if $g$ is of type $1A$, $2A$, $2B$, $3A$, $4A$, $4B$, $5A$, $6A$, $7AB$, $8A$. The input of the associated additive lift is the image of $(\varphi_d)_{d|N_g}$ under the isomorphism of Theorem \ref{th:iso}, where $\varphi_1=\varphi_g$ and $\varphi_d=0$ for $d>1$. When $g$ is of type $4C$ or $3B$, the following identity holds:
$$
\Phi_g \cdot \Borch(f_g) = \Grit(\varphi_g),
$$
where $f_g$ is a Jacobi form appearing in an identity established by Eguchi and Hikami \cite[\S 4]{EH12}
\begin{equation}
\phi_g(\tau,z) = -\frac{T_2\varphi_g(\tau,z)}{\varphi_g(\tau,z)} - f_g(\tau,z),    
\end{equation}
the form $\Borch(f_g)$ is the Borcherds multiplicative lift of the family $(\phi_1=f_g;\; \phi_d=0,\; d>1)$, and $\Grit(\varphi_g)$ is the Borcherds additive lift of the family $(\varphi_1=\varphi_g;\; \varphi_d=0,\; d>1)$. 
\end{theorem}

The paper is organized as follows. In Section 2 we collect necessary materials to prove the theorems, including the isomorphism between Jacobi forms and vector-valued modular forms, and the Jacobi forms representations of Borcherds additive lifts and Borcherds products for Siegel modular forms of degree two. Theorems \ref{mth1} and \ref{mth2} are proved in Section 3. We formulate the principal parts of the inputs of $\Phi_g$ and $\Borch(f_g)$ in Appendix A and Appendix B, respectively. 

\section{Preliminaries}
\subsection{Modular forms for the Weil representation}
Let $M$ be an even integral lattice of signature $(l^+,l^-)$ with dual lattice $M'$. The lattice $M$ induces a discriminant form $D_M:=(M'/M, Q)$ with the quadratic form
$$
Q : M'/M \rightarrow \mathbb{Q}/\mathbb{Z}, \; Q(x + M) = (x,x)/2 + \mathbb{Z}.
$$
Let $\mathrm{Mp}_2(\mathbb{Z})$ be the metaplectic group consisting of pairs $A = (A, \psi_A)$, where $A = \begin{psmallmatrix} a & b \\ c & d \end{psmallmatrix} \in \mathrm{SL}_2(\mathbb{Z})$ and $\psi_A$ is a holomorphic square root of $\tau \mapsto c \tau + d$ on $\mathbb{H}$.  The product in $\Mp_2(\ZZ)$ is 
$$
(A,\psi_A)(B,\psi_B) = (AB,\psi_A(B\tau)\psi_B(\tau)).
$$
The group $\Mp_2(\ZZ)$ is generated by $T = (\begin{psmallmatrix} 1 & 1 \\ 0 & 1 \end{psmallmatrix}, 1)$ and $S = (\begin{psmallmatrix} 0 & -1 \\ 1 & 0 \end{psmallmatrix}, \sqrt{\tau})$. 
The representation of $\mathrm{Mp}_2(\mathbb{Z})$ on the group ring $\mathbb{C}[D_M] = \mathrm{span}(\mathfrak{e}_x: \; x \in D_M)$ defined by 
$$
\rho_M(T) \mathfrak{e}_x = \mathbf{e}(-Q(x)) \mathfrak{e}_x \quad \text{and} \quad \rho_M(S) \mathfrak{e}_x = \frac{\mathbf{e}( \mathrm{sign}(M) / 8)}{\sqrt{|D_M|}} \sum_{y \in D_M} \mathbf{e}(( x,y)) \mathfrak{e}_y
$$
is called the \emph{Weil representation} $\rho_M$,
where $\mathbf{e}(t)=e^{2\pi it}$ for $t\in \CC$, and $\mathrm{sign}(M)=l^+ - l^- \mod 8$. 

A \emph{weakly holomorphic modular form} of weight $k \in \frac{1}{2}\mathbb{Z}$ for the Weil representation $\rho_M$ is a holomorphic function $f : \mathbb{H} \rightarrow \mathbb{C}[D_M]$ that satisfies 
$$
f(A \cdot \tau) = \psi_A(\tau)^{2k} \rho_M(A) f(\tau), \quad \text{for all $A \in \mathrm{Mp}_2(\mathbb{Z})$,}
$$ 
and which is meromorphic at infinity; that is, $f$ has a Fourier expansion of the form 
\begin{equation*}
f(\tau) = \sum_{x \in D_M} \sum_{\substack{n \in \mathbb{Z} - Q(x)\\ n \gg -\infty}} \alpha(n, x) q^n \mathfrak{e}_x.    
\end{equation*}
The finite sum  
$$
\sum_{x\in D_M}\sum_{n<0} \alpha(n, x) q^n \mathfrak{e}_x
$$
is called the \textit{principal part} of $f$. 
The form $f$ is called a \textit{holomorphic modular form} if it is holomorphic at infinity, i.e. its principal part is zero. The $\mathbb{C}$-vector spaces of weakly holomorphic and holomorphic modular forms of weight $k$ for $\rho_M$ are respectively denoted by $M_k^!(\rho_M)$ and $M_k(\rho_M)$.

\subsection{Jacobi forms}
Let $\Gamma \le \mathrm{SL}_2(\mathbb{Z})$ be a congruence subgroup. A Jacobi form of weight $k \in \mathbb{Z}$, index $m \in \mathbb{N}$, level $\Gamma$ and character $\chi : \Gamma \rightarrow \mathbb{C}^{\times}$ is a holomorphic function $\phi : \mathbb{H} \times \mathbb{C} \rightarrow \mathbb{C}$ satisfying the transformations $$\phi\left( \frac{a \tau + b}{c \tau + d}, \frac{z}{c\tau + d} \right) = (c \tau + d)^k \chi \left( \begin{pmatrix} a & b \\ c &  d \end{pmatrix} \right) \exp\Big( 2\pi i \frac{mcz^2}{c \tau + d} \Big) \phi(\tau, z)$$ and $$\phi(\tau, z + \lambda \tau + \mu) = \exp \Big( -2\pi i m \lambda^2 \tau - 4\pi i m \lambda z \Big) \phi(\tau, z)$$ for all $\begin{psmallmatrix} a & b \\ c & d \end{psmallmatrix} \in \Gamma$ and all $\lambda, \mu \in \mathbb{Z}$, as well as certain growth conditions at ``cusps". For any $A = \begin{psmallmatrix} a & b \\ c & d \end{psmallmatrix} \in \mathrm{SL}_2(\mathbb{Z})$, the Fourier expansion of $\phi$ at the cusp $A \cdot \infty$ is determined up to a root of unity through \begin{align*} \phi \Big|_{k, m} A(\tau) &= (c \tau + d)^{-k} e^{-2\pi i \frac{m c z^2}{c \tau + d}} \phi\left( \frac{a \tau + b}{c \tau + d}, \frac{z}{c \tau + d} \right) \\ &= \sum_{n, r} c_A(n, r) q^n \zeta^r, \end{align*} where $q = e^{2\pi i \tau}, \;\zeta = e^{2\pi i z}$ and $c_{A}(n, r) \in \mathbb{C}.$ We call $\phi$ a \emph{weakly holomorphic} Jacobi form if $c_A(n, r) = 0$ for all $A$ and all sufficiently small $n$; a \emph{weak} Jacobi form if $c_A(n, r) = 0$ whenever $n < 0$; and a \emph{holomorphic} Jacobi form if $c_A(n,r) = 0$ whenever $r^2 > 4mn$.

The spaces of weakly holomorphic, weak, and holomorphic Jacobi forms of weight $k$, index $m$ and level $\Gamma$ are written $$J_{k, m}^!(\Gamma, \chi), \quad J_{k,m}^{\w}(\Gamma, \chi), \quad J_{k, m}(\Gamma, \chi).$$ 

\subsection{The isomorphism between Jacobi forms and vector-valued modular forms}

Modular forms for the Weil representation have a useful equivalent description in terms of sequences of Jacobi forms. This was worked out for Jacobi forms of lattice index in \cite{WW22} and we describe the special case of scalar-index Jacobi forms below. 

Let $t$ and $N$ be positive integers. We fix the lattice $M := U(N) \oplus A_1(t)$, such that $M'/M \cong \mathbb{Z}/2t\mathbb{Z} \times \mathbb{Z}^2 / N \mathbb{Z}$. We write elements of $M'/M$ as tuples $(a/N, b/2t, c/N)$ with $a,c \in \mathbb{Z}/N\mathbb{Z}$ and $b \in \mathbb{Z}/2t\mathbb{Z}$, such that the quadratic form on $M'/M$ is $$Q(a/N, b/2t, c/N) = b^2/4t + ac/N.$$

\begin{theorem}\label{th:isomorphism} 
Let $k\in \ZZ$. There is an isomorphism 
$$
\hat{\mathbb{J}}:\; \bigoplus_{d | N} J^!_{k, t}(\Gamma_1(N/d))  \stackrel{\sim}{\longrightarrow} M^!_{k - \frac{1}{2}}(\rho_{M})
$$ 
which sends a sequence of weakly holomorphic Jacobi forms
$$
(\phi_d)_{d|N}, \quad \phi_d \in J_{k,t}^!(\Gamma_1(N/d)) \quad \text{for $d|N$}
$$
to the weakly holomorphic modular form of weight $k-\frac{1}{2}$ for $\rho_M$ defined as
\begin{align*}
\hat{\mathbb{J}}((\phi_d)_{d | N}) = \frac{1}{N} \sum_{a, b, c \in \mathbb{Z}/N\mathbb{Z}} e^{-2\pi i a c / N} \sum_{n \in \mathbb{Q}} \sum_{r \in \mathbb{Z}/2t\mathbb{Z}} c_{a, b}(n, r) q^{n - r^2 / 4t} \mathfrak{e}_{(c/N, r/2t, b/N)},
\end{align*}
where for any $(a, b) \in \mathbb{Z}^2$ the numbers $c_{a,b}(n,r)$ are defined by the Fourier expansion
$$
\phi_d \Big|_{k, t} A_{a,b} (\tau, z) = \sum_{n \in \mathbb{Q}} \sum_{r \in \mathbb{Q}} c_{a,b}(n, r) q^n \zeta^r,
$$
where $d = \mathrm{gcd}(a, b, N)$ and $A_{a,b} \in \mathrm{SL}_2(\mathbb{Z})$ with $A_{a, b} (a,b) = (d, 0)$ mod $N$. 
\end{theorem}
\begin{proof} This is a special case of \cite{WW22}. Note that $\hat{\mathbb{J}}$ is the inverse of the isomorphism labelled $\mathbb{J}$ there.
\end{proof}

The above map $\hat{\mathbb{J}}$ induces an isomorphism between spaces of holomorphic forms. 
The image of the subspace $\bigoplus_{d | N} J^!_{k, t}(\Gamma_0(N/d))$ under $\hat{\mathbb{J}}$ consists of those vector-valued modular forms $$F(\tau) = \sum_{x \in M'/M} f_x(\tau) \mathfrak{e}_x$$ with the property $f_{(a/N, b/2t, c/N)} = f_{(au/N, b/2t, cu^{-1}/N)}$ for every $u \in (\mathbb{Z}/N\mathbb{Z})^{\times}$.

\subsection{Siegel modular forms of degree two}
Let $\Gamma^{(2)}$ be an arithmetic subgroup of the symplectic group  $\mathrm{Sp}_4(\mathbb{R})$. An important class of examples for $\Gamma^{(2)}$ will be the following congruence subgroups of the paramodular group of level $t$:
$$
K_0(t,N)=\left\{ \begin{pmatrix} * & *t & * & * \\ * & * & * & */t \\ *N & *Nt & * & * \\ *Nt & *Nt & *t & * \end{pmatrix} \in \Sp_4(\QQ): \text{all $*$ are integers} \right\}.     
$$
Note that $K_0(1,1)=\Sp_4(\ZZ)$, $K_0(1,N)=\Gamma_0^{(2)}(N)$ is the congruence subgroup of $\Sp_4(\ZZ)$ appearing in Conjecture \ref{conj}, and $K_0(t,1)=K(t)$ is the usual paramodular group of level $t$. When $t>1$, the normal double extension $K_0(t,N)^+=K_0(t,N)\cup K_0(t,N)V_t$ also gives an example of $\Gamma^{(2)}$, where $V_t$ is the Fricke involution
$$
V_t=\frac{1}{\sqrt{t}}\begin{psmallmatrix} 0 & t & 0 & 0 \\ 1 & 0 & 0 & 0 \\ 0 & 0 & 0 & 1 \\ 0 & 0 & t & 0 \end{psmallmatrix}, \quad t\geq 1.
$$
$\Gamma^{(2)}$ acts on the Siegel upper half-space 
$$
\mathbb{H}_2 = \left\{Z = \begin{pmatrix} \tau & z \\ z & \omega \end{pmatrix} \in \mathrm{Mat}_2(\CC) : \mathrm{Im}(Z)>0\right\}
$$
by M\"obius transformations. A \emph{meromorphic Siegel modular form} of weight $k \in \frac{1}{2}\ZZ$ for $\Gamma^{(2)}$ with a character (or multiplier system) $\chi: \Gamma^{(2)} \to \CC^\times$ is a meromorphic function $F : \mathbb{H}_2 \rightarrow \mathbb{C}$ satisfying 
$$
F\left( A \cdot Z \right) = \chi(A) \mathrm{det}(cZ + d)^k F(Z), \quad A=\begin{pmatrix} a & b \\ c & d \end{pmatrix} \in \Gamma^{(2)}.
$$ 
If $V_t\in \Gamma^{(2)}$ then we have
$$
F\left( \begin{pmatrix} t\omega & z \\ z & \tau/t \end{pmatrix} \right) = (-1)^k\chi(V_t) F\left( \begin{pmatrix} \tau & z \\ z & \omega \end{pmatrix} \right).  
$$
    
\subsection{Singular Gritsenko lifts and Borcherds products}

In \cite{Bor98} Borcherds defined the singular theta lift and used it to find two constructions of meromorphic modular forms on orthogonal groups of lattices of signature $(l, 2)$. The first construction yields modular forms, called \emph{Borcherds products}, with special divisors and an infinite product expansion at every $0$-dimensional cusp. The second construction is a generalization of the Gritsenko lift which produces meromorphic modular forms of general weight $k \in \mathbb{N}$, each with poles of order exactly $k$ along special divisors. Both of these lifts take weakly holomorphic modular forms for the Weil representation as inputs.

In this subsection we want to describe the two singular theta lifts in special case of Siegel modular forms of degree two which can be identified with modular forms on orthogonal groups of signature $(3,2)$. Through the isomorphism $\hat{\mathbb{J}}$ in Theorem \ref{th:isomorphism}, the input forms into both lifts are understood as sequences of weakly holomorphic Jacobi forms.

The poles of singular Gritsenko lifts, and the zeros and poles of Borcherds products, are known explicitly and (for general lattices) lie on hyperplanes. Through the identification with the Siegel upper half-space these divisors become the classical Humbert surfaces. Let $a$, $n$, $r$, $m$ and $b$ be integers satisfying $\mathrm{gcd}(a,n,r,m,b)=1$ and
$$
\delta = \frac{ab}{N} - nm + \frac{r^2}{4t} >0.
$$
The Humbert surface (for $K_0(t, N)$) associated to this data is
$$
H(a,n,r,m,b)=\Big\{ \begin{pmatrix} \tau & z \\ z & \omega \end{pmatrix} \in \HH_2 : \frac{a}{N} + n\tau + rz + tm\omega + bt\cdot \det(Z) = 0 \Big\}.
$$
The number $\delta$ is called the discriminant of $H(a,n,r,m,b)$. Note that $K_0(t, N)$ maps Humbert surfaces to Humbert surfaces and preserves the discriminant.

We now describe the two theta lifts. We will only consider the special case that all Jacobi forms $\phi_d$ are $\Gamma_0(N/d)$-modular, because only this case is needed later. Under this assumption, the lifts are modular forms for the paramodular group $K_0(t, N)$ (possibly with characters). Additionally, the Fourier expansions of the lifts (but not the divisors) can be read off of the Fourier expansions of $\phi_d$ directly. Both theorems below are a special case of \cite{WW22} which describes theta lifts more generally in terms of Jacobi forms of lattice index. 

The Fourier--Jacobi expansions of the theta lifts involve certain Hecke operators. Let $\chi$ be a Dirichlet character modulo $D$ and $\phi \in J_{k, t}^!(\Gamma_0(D), \chi)$ be a weakly holomorphic Jacobi form with Fourier expansion at infinity 
$$
\phi(\tau, z) = \sum_{n \in \mathbb{Z}} \sum_{r \in \mathbb{Z}} c(n, r) q^n \zeta^r.
$$ 
For any positive integer $m$ the \emph{index-raising Hecke operators} $T^{(D)}_-(m)$ are defined by 
$$
\phi \Big|_{k, t} T_{-}^{(D)}(m)(\tau, z) = \frac{1}{m} \sum_{\substack{a, d \in \mathbb{N} \\ ad = m \\ (a, D) = 1}} \sum_{b \in \mathbb{Z}/d\mathbb{Z}} a^k \chi(a) \phi \Big( \frac{a \tau + b}{d}, az \Big).
$$ 
In particular 
$$
\phi \Big|_{k, t} T_{-}^{(D)}(m)(\tau, z) = \sum_{n \in \mathbb{Z}} \sum_{r \in \mathbb{Z}} \sum_{\substack{a \in \mathbb{N} \\ (a, D) = 1 \\ a | \mathrm{gcd}(n,r,m)}} a^{k-1} \chi(a) c \left( \frac{n m}{a^2}, \frac{r}{a} \right) q^n \zeta^r.
$$ 
By \cite[Lemma 2.1]{CG11}, we have
$$
\phi \Big|_{k, t} T_{-}^{(D)}(m) \in J_{k, mt}^!(\Gamma_0(D), \chi), \quad m\geq 1.
$$
Following \cite[Theorems 7.1, 9.3]{Bor95} and \cite[Theorem 14.3]{Bor98}, when $k\geq 0$ we define the index-zero Hecke operator by 
$$
\phi \Big|_{k, t} T_{-}^{(D)}(0)(\tau, z) = c_{\phi} + \sum_{\substack{n,r\in\ZZ\\ (n,r)>0}} \sum_{\substack{a \in \mathbb{N} \\ (a, D) = 1 \\ a | \mathrm{gcd}(n, r)}} a^{k-1} \chi(a) c(0, r/a) q^n \zeta^r,
$$ 
where $c_{\phi}$ is a certain constant (depending on $\phi$, involving Bernoulli numbers), such that the above sum defines a meromorphic Jacobi form of weight $k$ and index $0$ for $\Gamma_0(D)$ with character $\chi$. If all coefficients $c(0,r)$ of $\phi$ are zero, then $\phi \big|_{k, t} T_-^{(D)}(0)$ is identically zero. 

Let $\phi \in J_{k, t}^!(\Gamma_0(D),\chi)$ with $k\geq 0$. We define the \emph{formal additive lift} of $\phi$ as
$$
\mathbf{G}(\phi)\left( \begin{pmatrix} \tau & z \\ z & w \end{pmatrix} \right) = \sum_{m=0}^{\infty} \phi \Big|_{k, t} T_-^{(D)}(m)(\tau, z) e^{2\pi imt\omega}.
$$

\begin{theorem}\label{th:additive} 
Let $M = U(N) \oplus A_1(t)$ with quadratic form $Q$. Let $\phi \in J_{k, t}^!(\Gamma_0(N),\chi)$, where $k$ is a positive integer and $\chi$ is a Dirichlet character modulo $N$. Then the additive lift $\Grit(\phi)$ defines a meromorphic Siegel modular form of weight $k$ for $K_0(t,N)^+$ with a character induced by $\chi$, and this form is always symmetric, i.e. $\Grit(\phi)(V_t\cdot Z)=\Grit(\phi)(Z)$. 

We define $\phi_1=\phi$, $\phi_d=0$ for $d>1$ and write 
$$
\hat{\mathbb{J}}(\phi):=\hat{\mathbb{J}}((\phi_d)_{d | N}) = \sum_{\gamma \in M'/M} \sum_{n \in \mathbb{Z} - Q(\gamma)} \alpha(n, \gamma) q^n \mathfrak{e}_{\gamma}.
$$ 
Then the only singularities of $\Grit(\phi)$ are poles of order $k$ along the Humbert surfaces $H(a,n,r,m,b)$ of discriminant $\delta=ab/N-nm+r^2/4t$ for which 
$$
\sum_{\lambda=1}^\infty \lambda^{-k} \alpha\left(-\lambda^2\delta, \Big( \frac{\lambda a}{N}, \frac{\lambda r}{2t}, \frac{\lambda b}{N} \Big)\right) \neq 0. 
$$
\end{theorem}
\begin{proof}
By \cite[Corollary 5.2]{WW22}, the vector-valued modular form $\hat{\mathbb{J}}(\phi)$ is invariant under $(\ZZ/N\ZZ)^\times$ up to some character induced by $\chi$. It is known that $K_0(t,N)/\{\pm I\}$ is isomorphic to the group generated by $(\ZZ/N\ZZ)^\times$ and the discriminant kernel of the signature $(3,2)$ lattice $U(N)\oplus U\oplus A_1(t)$ (see e.g. \cite[\S 2]{HWK21}). Therefore, we obtain the modularity of $\Grit(\phi)$ by applying \cite[Remark 5.15]{WW22} to the rank-one lattice $L=A_1(t)$. The property of singularities follows from \cite[Theorem 14.3]{Bor98}.
\end{proof}

\begin{theorem}\label{th:FJ-level-N}
Let $M = U(N) \oplus A_1(t)$ with quadratic form $Q$. 
Let $\phi_d \in J_{0, t}^!(\Gamma_0(N/d))$, $d | N$ be weakly holomorphic Jacobi forms of weight zero with Fourier expansions 
$$
\phi_d(\tau, z) = \sum_{n \in \mathbb{Z}} \sum_{r \in \mathbb{Z}} c_d(n, r) q^n \zeta^r,
$$ 
and suppose the principal part of
$$
\hat{\mathbb{J}}((\phi_d)_{d|N}) = \sum_{\gamma \in M'/M} \sum_{n \in \mathbb{Z}-Q(\gamma)} \alpha(n, \gamma) q^n \mathfrak{e}_{\gamma}
$$ 
is integral, i.e. $\alpha(n,r)\in\ZZ$ for $n<0$.  Then there is a meromorphic modular form $\Psi=\Borch((\phi_d)_{d | N})$ for the group $K_0(t, N)^+$, the \emph{Borcherds product}, with the following properties: 
\begin{itemize}
\item[(1)] When $\tau$ and $w$ have sufficiently large imaginary part, $\Psi$ is given by the infinite product 
\begin{align*} 
\Psi(Z) &= q^A \zeta^B s^C \prod_{(n,r, m) > 0} (1 - q^n \zeta^{r} s^{tm})^{\mathrm{mult}(n,r,m)} \\ &= q^A \zeta^B s^C \prod_{d | N} \prod_{(n, r, m) > 0} (1 - q^{dn} \zeta^{dr} s^{dtm})^{\mathrm{mult}_d(nm, r)} 
\end{align*} 
where $s=e^{2\pi i\omega}$ and $(n,r,m) > 0$ means either $n > 0$, or $n=0$ and $m > 0$, or $n=m=0$ and $r < 0$;  let $\mu$ denote the M\"obius function;
$$
\mathrm{mult}(n, r, m) = \sum_{\substack{bd > 0 \\ bd | \mathrm{gcd}(n,r,m,N)}} \frac{\mu(b)}{bd} c_d \left( \frac{nm}{b^2 d^2}, \frac{r}{bd} \right);$$ 
$$
\mathrm{mult}_d(n, r) = \sum_{t | d} \frac{\mu(d/t)}{d} c_t(n, r);
$$ 
and where $$A = \frac{1}{24} \sum_{r \in \mathbb{Z}} c_N(0,r), \quad B = \frac{1}{2} \sum_{r > 0} c_N(0, r) r, \quad C = \frac{1}{4} \sum_{r \in \ZZ} c_N(0, r) r^2$$ is the \emph{Weyl vector}.

\item[(2)] $\Psi$ has the Fourier--Jacobi expansion 
$$
\Psi(Z) = \Theta(\tau, z) s^C \cdot \exp \Big( - \sum_{d | N} d^{-1} \mathbf{G}(\phi_d)(dZ) \Big),
$$ 
in which the leading Fourier--Jacobi coefficient $\Theta$ is the generalized theta block 
$$
\Theta(\tau, z) = \prod_{d | N} \Big[ \eta(d \tau)^{\mathrm{mult}_d(0,0)} \prod_{r=1}^{\infty}  \Big(\frac{\vartheta(d \tau, drz)}{\eta(d \tau)} \Big)^{\mathrm{mult}_d(0,r)} \Big],
$$ 
where as usual 
$$
\eta(\tau) = q^{1/24} \prod_{n=1}^{\infty} (1 - q^n) = \sum_{n=1}^{\infty} \left( \frac{12}{n} \right) q^{n^2 / 24}
$$ 
and 
\begin{align*}
\vartheta(\tau, z) &= q^{1/8} (\zeta^{1/2} - \zeta^{-1/2}) \prod_{n=1}^{\infty} (1 - q^n \zeta)(1 -q^n)(1-q^n \zeta^{-1})\\
&= \sum_{n=1}^{\infty} \left( \frac{4}{n} \right) q^{n^2 / 8} (\zeta^{n/2} - \zeta^{-n/2})    
\end{align*}

\item[(3)] $\Psi$ has weight 
$$
\frac{\alpha(0,0)}{2} = \frac{1}{2} \sum_{d | N} \mathrm{mult}_d(0,0).
$$

\item[(4)] All poles and zeros of $\Psi$ lie on Humbert surfaces. The multiplicity of $\Psi$ along the Humbert surface $H(a,n,r,m,b)$ of discriminant $\delta$ is the sum 
$$
\sum_{\lambda=1}^\infty \alpha\left(-\lambda^2\delta, \Big( \frac{\lambda a}{N}, \frac{\lambda r}{2t}, \frac{\lambda b}{N} \Big)\right). 
$$

\item[(5)] We have the duality 
$$
\Psi\left( \begin{psmallmatrix} t\omega & z \\ z & \tau/t    
\end{psmallmatrix} \right)= (-1)^{D_0} \Psi\left( \begin{psmallmatrix} \tau & z \\ z & \omega    
\end{psmallmatrix} \right),
$$
where $D_0$ is the multiplicity of $\Psi$ along the divisor $H(0,-1,0,1,0)$, i.e.
$$
D_0 = \frac{1}{N} \sum_{d|N} \sum_{\lambda=1}^\infty \varphi(N/d)c_d(-\lambda^2,0),
$$
here $\varphi(-)$ is the Euler’s totient function. 
\end{itemize}
\end{theorem}
\begin{proof}
This is a special case of \cite[Theorem 5.9]{WW22}.     
\end{proof}

If every $\phi_d$ has rational Fourier expansion at infinity then $\hat{\mathbb{J}}((\phi_d)_{d|N})$ also has rational Fourier expansion by \cite[Lemma 5.6]{WW22}. Therefore, there exists an integer $D$ such that $\hat{\mathbb{J}}((D\cdot\phi_d)_{d|N})$ has integral principal part. 

We also have the following expression for $\Psi$ which is formally similar to the infinite product proposed by Cheng--Duncan (Equation \eqref{eq:infprod} of the introduction). 

\begin{proposition}[(5.7) in \S 5.2 of \cite{WW22}]\label{prop:exp-product} Suppose the conditions of Theorem \ref{th:FJ-level-N} hold. For any positive integer $a$ we define
\begin{equation*}
\phi_a = \phi_{d}, \quad d=\mathrm{gcd}(a,N).    
\end{equation*}
Then the Borcherds product $\Psi$ is represented on an open subset of $\mathbb{H}_2$ by the infinite product
\begin{equation}
\Psi(Z) = q^A\zeta^B s^C \prod_{(n,r,m)>0} \exp\left( -\sum_{a=1}^\infty c_{a}(nm,r) \frac{(q^n\zeta^r s^{tm})^a}{a}  \right).            
\end{equation}
\end{proposition}

\begin{remark} Cl\'ery--Gritsenko (\cite{CG11}, especially Theorem 3.1) have considered Borcherds products on the groups $K_0(t, N)$ associated to Jacobi forms. In our notation, their result concerns the special case that all $\phi_d$, $d > 1$ are traces of $\phi_1$, i.e. $$\phi_d = \sum_{M \in \Gamma_0(N) \backslash \Gamma_0(N/d)} \phi_1 \Big|_{k, t} M.$$
\end{remark}

\subsection{The Mathieu group and eta products}

The largest Mathieu group $M_{24}$ is a sporadic simple group of order $$244823040 = 2^{10} \cdot 3^3 \cdot 5 \cdot 7 \cdot 11 \cdot 23$$ with a natural linear action on the Leech lattice, the unique unimodular lattice of rank 24 without roots. Through this action, every  $g \in M_{24}$ has a characteristic polynomial of the form $$\mathrm{det}(X - g) = \prod_{k=1}^{\infty} (1 - X^k)^{b_k}, \quad b_k \in \mathbb{N}_0.$$ The symbol $\prod_{b_k \neq 0} k^{b_k}$ is called the \emph{cycle shape} of $g$ and it depends only on the conjugacy class of $g$ in $M_{24}$.

There are $26$ conjugacy classes in $M_{24}$ with $21$ distinct cycle shapes (see e.g. \cite[Table 1]{CD12}). For a conjugacy class $[g]$ with cycle shape $\prod k^{b_k}$, the largest $k \in \mathbb{N}$ for which $b_k \neq 0$ is the order of $g$, denoted $n_g$. The product of $n_g$ and the smallest $k$ for which $b_k \neq 0$ is the \emph{level} of $g$, denoted $N_g$; it is the level of the associated eta product $$\eta_g(\tau) = \prod_k \eta(k \tau)^{b_k} = q^{\frac{1}{24} \sum_k k b_k} \prod_{n=1}^{\infty} \prod_{k=1}^{\infty} (1 - q^{nk})^{b_k}.$$ More precisely, $\eta_g$ is a cusp form of weight $\frac{1}{2} \sum_k b_k$ for the group $\Gamma_0(N_g)$ and a Dirichlet character, which is trivial if the weight is even. 

We write $\chi(g) := \mathrm{tr}(g)$ for the trace of $g$ acting on the Leech lattice; if $g$ has cycle shape $\prod_k k^{b_k}$ then $\chi(g) = b_1$. For any $d \in \mathbb{N}$, the cycle shape of $g^d$ is $$\prod_{k=1}^\infty \Big( k/\mathrm{gcd}(k,d) \Big)^{\mathrm{gcd}(k,d)\cdot b_k},$$ and in particular $\chi(g^d) = \sum_{k | d} k b_k$. By M\"obius inversion, the exponents in the cycle shape are
\begin{equation} \label{eq:b_m}
b_d = \frac{1}{d} \sum_{k | d} \mu(d/k) \chi(g^k). 
\end{equation}

Each conjugacy class $[g]$ in $M_{24}$ determines a twisted elliptic genus of $K3$ surfaces. This is a weak Jacobi form of weight $0$ and index $1$ and level $\Gamma_0(N_g)$ with trivial character, given by the expression 
\begin{align*}
\phi_g(\tau, z) &= \frac{\chi(g)}{12} \phi_{0, 1}(\tau, z) + \tilde{T}_g(\tau) \phi_{-2, 1}(\tau, z) \\ &=\sum_{n=0}^{\infty} \sum_{r \in \mathbb{Z}} c_g(4n - r^2) q^n \zeta^r \\ &= 2 \zeta^{-1} + (\chi(g) - 4) + 2 \zeta + O(q). 
\end{align*} 
Here, $\phi_{-2, 1}$ and $\phi_{0, 1}$ are the basic weak Jacobi forms
\begin{align*}
\phi_{-2, 1}(\tau, z) &= \frac{\vartheta(\tau, z)^2}{\eta(\tau)^6} = \zeta^{-1} - 2 + \zeta + O(q) \in J_{-2, 1}^{\w}(\mathrm{SL}_2(\mathbb{Z})),  \\
\phi_{0, 1}(\tau, z) &= -\frac{3}{\pi^2} \wp(\tau, z) \cdot \phi_{-2, 1}(\tau, z) = \zeta^{-1} + 10 + \zeta + O(q) \in J_{0, 1}^{\w}(\mathrm{SL}_2(\mathbb{Z})),
\end{align*}
where $\wp$ is the Weierstrass $\wp$-function. The function $\tilde T_g$ is a particular modular form of weight $2$ which can be found in \cite[Table 2]{EH12} or \cite[Table 1]{GS21}.

\begin{lemma}\label{lem:Zg} Fix $g \in M_{24}$.
Let $a$ be a positive integer and $d = \mathrm{gcd}(a, n_g)$. Then $\phi_{g^a} = \phi_{g^d}$. Moreover, for any $d|N_g$, $\phi_{g^d}$ is a weak Jacobi form on $\Gamma_0(N_g/d)$ with trivial character.
\end{lemma}
\begin{proof}
The form $\phi_h$ depends only the cycle shape of $h$. Since $g^a$ and $g^d$ have the same cycle shape, the equality $\phi_{g^a} = \phi_{g^d}$ is proved. To prove the last claim it suffices to show that the level of $g^d$ divides $N_g/d$. This is verified by direct calculation. 
\end{proof}

\section{Proofs of the main theorems}
We first employ Theorem \ref{th:FJ-level-N} to prove Conjecture \ref{conj}, that is, $\Phi_g$ is a meromorphic Siegel modular form on $\Gamma_0^{(2)}(N_g)$ which can be constructed as a Borcherds product on $U(N_g)\oplus U\oplus A_1$. 

\begin{proof}[Proof of Theorem \ref{mth1}]
It is clear from the definition that each $\phi_g$ has rational Fourier coefficients at infinity. By \cite[Lemma 5.6]{WW22}, the image of $(\phi_{g^d})_{d|N_g}$ under the isomorphism $\hat{\mathbb{J}}$ (see Theorem \ref{th:isomorphism}), which we denote $F_g$, also has rational Fourier expansion. Therefore, there exists a denominator $d_g \in \mathbb{N}$ such that $d_g \cdot F_g$ has integral principal part. It follows that $\Borch(d_g F_g)$ is a well-defined meromorphic Borcherds product. 
Since 
$$
\phi_{g^{N_g}}(\tau,z) = 2\phi_{0,1}(\tau,z) = 2(\zeta+\zeta^{-1}+10)+O(q),
$$
the Weyl vector of $\Borch(d_g F_g)$ is $(A,B,C) = (d_g, d_g, d_g)$. Using Proposition \ref{prop:exp-product} and Lemma \ref{lem:Zg} we see that $$\Borch(d_g F_g) = \Phi_g^{d_g}.$$ By Theorem \ref{th:FJ-level-N}, the weight of $\Phi_g$ is 
\begin{align*}
k_g&=\frac{1}{2}\sum_{d|N_g} \mathrm{mult}_d(0,0)  \\ &=\frac{1}{2} \sum_{d|N_g} \frac{1}{d} \sum_{k|d} \mu(d/k) c_{g^k}(0) \\
& = \frac{1}{2} \sum_{d|N_g} \frac{1}{d} \sum_{k|d} \mu(d/k) \mathrm{tr}(g^k) - 2\sum_{d|N_g} \frac{1}{d} \sum_{k|d} \mu(d/k) \\
&= -2 + \frac{1}{2}\sum_{k} b_k,
\end{align*}
where we have used Equation \eqref{eq:b_m} in the last equality.

It remains to show that $\Borch(F_g)$ itself is single-valued, i.e. the principal part of $F_g$ is integral. By Theorem \ref{th:isomorphism}, the principal part of $F_g$ is completely determined by the Fourier coefficients of index $(n,r)$ with $n<\frac{1}{4}$ of $\phi_{g^d}$ at all cusps for all $d|n_g$. The forms $\tilde{T}_g$ are expressions in eta quotients and modular forms of the type $$
E_2^{(N)}(\tau):=\frac{1}{N-1}\Big(NE_2(N\tau) - E_2(\tau)\Big)=1+O(q),
$$ where $E_2(\tau) = 1 - 24 \sum_{n=1}^{\infty} \sigma_1(n) q^n$. The Fourier expansion of $E_2^{(N)}$ at any cusp is easy to determine using the fact that $E_2$ is quasimodular for the full group $\mathrm{SL}_2(\mathbb{Z})$, and the Fourier expansion of an eta quotient at any cusp was given explicitly by Scheithauer \cite[Propositions 6.1 and 6.2]{Sch09}. Using these formulas, we were able to find the principal parts by a computer calculation; they are listed in Appendix A and are indeed integral. We have thus proved that $\Phi_g$ is a meromorphic Siegel modular form of weight $k_g$ for $\Gamma_0^{(2)}(N_g)$. 
\end{proof}

Theorem \ref{mth1} yields short proofs of three previously-known results which have applications to mathematical physics (see e.g. \cite{Che10, CD12, PV14}):

\begin{corollary}
The form $\Phi_g$ has the Fourier--Jacobi expansion
\begin{equation}
\Phi_g(Z) = \varphi_g(\tau,z)\cdot s \cdot \exp\Big( -\sum_{d|N_g} \Grit\big(d^{-1}\phi_{g^d}\big)(dZ) \Big),    
\end{equation}
where the leading Fourier--Jacobi coefficient $\varphi_g$ is
\begin{equation}
    \varphi_g(\tau,z) = \eta_g(\tau)\phi_{-2,1}(\tau,z).
\end{equation}
\end{corollary}
\begin{proof}
Except for the explicit description of $\varphi_g$, this follows immediately from Theorem \ref{th:FJ-level-N}. The Jacobi form $\varphi_g$ is a theta block involving the coefficients $\mathrm{mult}_d(0,0)$ and $\mathrm{mult}_d(0,1)$ for $d|N_g$. We find
\begin{align*}
\mathrm{mult}_d(0,1) &= \frac{2}{d}\sum_{t|d} \mu(d/t) = \begin{cases} 2: & d = 1; \\ 0: & d > 1; \end{cases}
\end{align*}
and similarly
\begin{align*}
\mathrm{mult}_d(0,0) &= \frac{1}{d}\sum_{t|d} \mu(d/t) \mathrm{tr}(g^t) -  \frac{4}{d}\sum_{t|d} \mu(d/t)  = \begin{cases} b_1 - 4: & d = 1; \\ b_d: & d > 1. \end{cases}
\end{align*}
Therefore, 
\[
\varphi_g(\tau,z) = \frac{\prod_{d|N_g} \eta(d\tau)^{b_d}}{\eta(\tau)^4} \cdot \frac{\vartheta(\tau,z)^2}{\eta(\tau)^2} = \eta_g(\tau)\phi_{-2,1}(\tau,z). \qedhere
\]
\end{proof}

\begin{corollary}
The following duality holds:
\begin{equation}
\Phi_g(\tau,z,\omega) = \Phi_{g}(\omega, z,\tau).     
\end{equation}
\end{corollary}
\begin{proof}
This follows from Theorem \ref{th:FJ-level-N} (5) and the number $D_0=0$. 
\end{proof}

\begin{corollary}
The following identity holds:
\begin{equation}
\lim_{z\to 0}\frac{\Phi_g(\tau, z, w)}{(2\pi iz)^2} = \eta_g(\tau)\eta_g(\omega).    
\end{equation}
\end{corollary}
\begin{proof}
Note that $\Phi_g$ has a double zero at $z=0$. The limit can be viewed as the quasi-pullback of $\Phi_g$ along the embedding $$U(N_g)\oplus U \hookrightarrow U(N_g)\oplus U\oplus A_1.$$ This is itself the Borcherds lift of the sequence of Jacobi forms of index $\{0\}$ (i.e. modular forms),  $$\phi_{g^d}(\tau, 0), \quad d | N_g.$$ Here $\phi_g(\tau,0)=\chi(g)=\mathrm{tr}(g)$. The identity now follows from \cite[Remark 5.12]{WW22}. 
\end{proof}

Now we investigate when $\Phi_g$ is an additive lift. Eguchi and Hikami \cite{EH12} have represented the twisted elliptic genus $\phi_g$ in terms of the associated eta product $\eta_g$. More precisely, they find that
\begin{equation}
\phi_g(\tau,z) = - \frac{\varphi_g\Big|_{k_g,1}T_{-}^{(N_g)}(2)(\tau,z)}{\varphi_g(\tau,z)}  - f_g(\tau,z),
\end{equation}
where $\varphi_g = \eta_g \phi_{-2, 1}$ is the leading Fourier--Jacobi coefficient and where $f_g$ is identically zero if and only if $g$ is of type $1A$, $2A$, $2B$, $3A$, $4A$, $4B$, $5A$, $6A$, $7AB$, $8A$. 
In addition, they found
\begin{align}
f_{4C}(\tau,z)&=16\frac{\eta(2\tau)^4\eta(8\tau)^4}{\eta(4\tau)^4}\phi_{-2,1}(\tau,z),\\
f_{3B}(\tau,z)&=18\frac{\eta(\tau)^3\eta(9\tau)^3}{\eta(3\tau)^2}\phi_{-2,1}(\tau,z).
\end{align}
These $12$ cycle shapes correspond to the only conjugacy classes of $M_{24}$ for which $\varphi_g$ has positive weight. For each of these, by Theorem \ref{th:additive} we have 
\begin{equation}
\Grit(\varphi_g)(Z) = \sum_{m=1}^\infty \varphi_g\Big|_{k_g,1}T_{-}^{(N_g)}(m)(\tau,z)\cdot s^m, \quad s=e^{2\pi i\omega}.    
\end{equation}
Let $\hat{F}_g$ be the $\hat{\mathbb{J}}$-image of $(\phi_d)_{d|N_g}$ where $\phi_1=f_g$ and $\phi_d=0$ for $d>1$. By Theorem \ref{th:FJ-level-N},
\begin{equation}
\Borch(\hat{F}_g)(Z)= \exp\left( -\sum_{m=1}^\infty f_g\Big|_{0,1}T_{-}^{(N_g)}(m)(\tau,z)\cdot s^m \right) 
\end{equation}
has weight $0$. 
Note that $\Borch(\hat{F}_g)=1$ if $f_g=0$. 

Theorem \ref{mth2} is equivalent to the following result. 

\begin{theorem}\label{mth2-equiv}
Let $[g]$ be a conjugacy class of $M_{24}$ with positive-weight $\varphi_g$. Then we have
\begin{equation}
    \Phi_g(Z)\cdot \Borch(\hat{F}_g)(Z) = \Grit(\varphi_g)(Z).
\end{equation}
\end{theorem}
\begin{proof}
By Koecher's principle it is enough to show that $\frac{\Grit(\varphi_g)}{\Phi_g \cdot \Borch(\hat{F}_g)}$ is holomorphic, since $\Grit(\varphi_g)$ and $\Phi_g \cdot \Borch(\hat{F}_g)$ have the same weight and the same leading Fourier--Jacobi coefficient $\varphi_g$. The divisors of $\Phi_g$ and $\mathbf{B}(\hat F_g)$ can be read off of the data in the appendices. 

When $g$ is of type $1A$, $2A$, $3A$, $4B$, the product $\Phi_g$ has only a double zero on the diagonal $z=0$. On the other hand, the double zero of $\varphi_g(\tau, z)$ at $z = 0$ leads to a double zero of $\Grit(\varphi_g)$ at $z = 0$, and we are done. 

The other cases require more work. Firstly, the Gritsenko lift of $\varphi_g = \eta_g \phi_{-2, 1}$ may have poles, but these are easily described. If the theta-decomposition of $\phi_{-2, 1}$ is written $f_0 \mathfrak{e}_0 + f_{1/2} \mathfrak{e}_{1/2}$, where $$f_0(\tau) = 10 + 108 q + 808 q^2 + 4016q^3 + ...$$ $$f_{1/2}(\tau) = q^{-1/4} - 64q^{3/4} - 513 q^{7/4} - 2752 q^{11/4} - ...$$ then we have $$\hat{\mathbb{J}}(\varphi_g) = \sum_{\substack{c, d \in \mathbb{Z}/N_g\mathbb{Z} \\ \mathrm{gcd}(c, d, N_g) = 1}} \Big( \eta_g \Big| \begin{pmatrix} * & * \\ c & d \end{pmatrix} \Big) \big(f_0 \cdot \mathfrak{e}_{(c/N_g, 0, d/N_g)} + f_{1/2} \cdot \mathfrak{e}_{(c/N_g, 1/2, d/N_g)} \big),$$ so the principal part of $\hat{\mathbb{J}}(\varphi_g)$ can be determined from the cusps of $\Gamma_0(N_g)$ at which $\eta_g$ vanishes to order strictly less than $1/4$. 

Secondly, $\Phi_g\cdot \mathbf{B}(\hat F_g)$ generally has zeros on Humbert surfaces $r^{\perp}$ other than the diagonal $z=0$, and $\Grit(\varphi_g)$ must be shown to vanish on these zeros to at least the correct order. Following section 5 of \cite{W21}, to show that $\Grit(\varphi_g)$ vanishes on $r^{\perp}$ to order $N$ it is sufficient to show that the ``development coefficients" $D_k^{r^{\perp}} \hat{\mathbb{J}}(\varphi_g)$ of $\hat{\mathbb{J}}(\varphi_g)$ along $r^{\perp}$ vanish for $0 \le k \le N$. The latter are weakly holomorphic vector-valued modular forms for $\mathrm{SL}_2(\mathbb{Z})$ and they vanish identically as soon as all Fourier coefficients are zero up to an appropriate Sturm bound. 

As soon as we have shown that every pole of $\Grit(\varphi_g)$ appears as a pole of $\Phi_g \cdot \Borch(\hat{F}_g)$, and that every zero of $\Phi_g \cdot \Borch(\hat{F}_g)$ appears as zero of $\Grit(\varphi_g)$, we can apply Koecher's principle as in the first paragraph. All of the steps above can be carried out using the Sage package \cite{Wil22}.

As an example, we will describe the computations for the conjugacy class $3B$ of level $9$. The form $\eta_g(\tau) = \eta(3\tau)^8$ vanishes at the cusps $1/3, 2/3, \infty$ of $\Gamma_0(9)$ to order $1$, but it vanishes at the cusp $0$ to order only $1/9$; indeed, $$\tau^{-4} \eta_g(-1/\tau) = \frac{1}{81} \eta(\tau/3)^8 = \frac{1}{81} q^{1/9} - \frac{8}{81} q^{4/9} \pm ...$$ Therefore the associated vector-valued modular form has Fourier expansion beginning \begin{align*} \hat{\mathbb{J}}(\varphi_g) =& \frac{1}{81} q^{-5/36} \sum_{a \in (\mathbb{Z}/9\mathbb{Z})^{\times}} \mathfrak{e}_{(a/9, 1/2, -a^{-1}/9)} - \frac{2}{81} q^{1/9} \sum_{a \in (\mathbb{Z}/9\mathbb{Z})^{\times}} \mathfrak{e}_{(a/9, 0, -a^{-1}/9)} \\ &- \frac{8}{81} q^{7/36} \sum_{a \in (\mathbb{Z}/9\mathbb{Z})^{\times}} \mathfrak{e}_{(a/9, 1/2, 5a^{-1}/9)} + \frac{16}{81} q^{4/9} \sum_{a \in (\mathbb{Z}/9\mathbb{Z})^{\times}} \mathfrak{e}_{(a/9,0, 5a^{-1}/9)} \\ &+ O(q^{19/36}), \end{align*}  where $a^{-1}$ is the inverse of $a$ modulo $9$. This shows that $\Grit(\varphi_g)$ has double poles only on Humbert surfaces of discriminant $5/36$. Consulting Appendix A and Appendix B shows that $\Phi_g$ is holomorphic there and $\Borch(\hat{F}_g)$ contributes double poles on those divisors as well. We also see that $\Phi_g \cdot \Borch(\hat{F}_g)$ has only double zeros on the Humbert surfaces of discriminants $1/4$ and $1/36$. We verify that the first two development coefficients of $\hat{\mathbb{J}}(\varphi_g)$ along these divisors vanish to at least the respective Sturm bounds $\frac{1}{12} \cdot 2 = 1/6$ and $\frac{1}{12} \cdot 3 = 1/4$ using the Fourier expansion above, and we are done.
\end{proof}

For the other conjugacy classes, the form $\varphi_g$ has non-positive weight and there seems to be no reasonable way to define its additive lift.

\bigskip
\noindent
\textbf{Acknowledgements} 
H. Wang is supported by the Institute for Basic Science (IBS-R003-D1).  

\section*{Appendix A: Data for Theorem \ref{mth1}}

For each cycle shape of level $N_g$, the preimage of $\Phi_g$ under the Borcherds lift is expressed as a sequence of weak Jacobi forms $(\phi_d)_{d|N_g}$. We list the input forms $\phi_d$ as well as the principal part of the vector-valued modular form $\hat{\mathbb{J}}((\phi_d)_{d | N_g})$.

\bigskip

\noindent \textbf{Class 1A.} cycle shape $1^{24}$, order $1$, level $1$. The input Jacobi form is \begin{align*} \phi_1 &:= \phi_{1A}(\tau, z) \\ &= 2 \zeta^{\pm 1} + 20 + (20 \zeta^{\pm 2} - 128 \zeta^{\pm 1} + 216)q + (2 \zeta^{\pm 3} + 216 \zeta^{\pm 2} - 1026 \zeta^{\pm 1} + 1616)q^2 + O(q^3). \end{align*} 

Principal part: $$20 \mathfrak{e}_{(0,0,0)} + 2 q^{-1/4} \mathfrak{e}_{(0,1/2,0)}.$$

\bigskip

\noindent \textbf{Class 2A.} cycle shape $1^{8}2^8$, order $2$, level $2$. The input Jacobi forms are \begin{align*} \phi_1 &:= \phi_{2A}(\tau, z) \\ &= 2 \zeta^{\pm 1} + 4 + (4 \zeta^{\pm 2} - 8) q + (2 \zeta^{\pm 3} - 8 \zeta^{\pm 2} - 2 \zeta^{\pm 1} + 16) q^2 + O(q^3) \end{align*} and $\phi_2 = \phi_{1A}$.

Principal part: $$12 \mathfrak{e}_{(0,0,0)} + 2 q^{-1/4} \mathfrak{e}_{(0,1/2,0)}.$$

\bigskip

\noindent \textbf{Class 2B.} cycle shape $2^{12}$, order $2$, level $4$. The input Jacobi forms are \begin{align*} \phi_1 &:= \phi_{2B}(\tau, z) \\ &= 2 \zeta^{\pm 1} - 4+ (-4 \zeta^{\pm 2} + 8) q + (2 \zeta^{\pm 3} + 8 \zeta^{\pm 2} - 2 \zeta^{\pm 1} - 16) q^2 + O(q^3) \end{align*} and $\phi_2 = \phi_4 = \phi_{1A}$. 

Principal part: $$8 \mathfrak{e}_{(0,0,0)} + 2 q^{-1/4} \mathfrak{e}_{(0, 1/2, 0) } + 2 q^{-1/4} \mathfrak{e}_{(1/2,1/2,1/2)}.$$

\bigskip

\noindent \textbf{Class 3A.} cycle shape $1^{6}3^6$, order $3$, level $3$. The input Jacobi forms are \begin{align*} \phi_1 &:= \phi_{3A}(\tau, z) \\ &= 2 \zeta^{\pm 1} + 2 + (2 \zeta^{\pm 2} - 2 \zeta^{\pm 1})q + (2 \zeta^{\pm 3} - 4) q^2 + O(q^3) \end{align*} and $\phi_3 = \phi_{1A}$.

Principal part: $$8 \mathfrak{e}_{(0,0,0)} + 2 q^{-1/4} \mathfrak{e}_{(0,1/2,0)}.$$

\bigskip

\noindent \textbf{Class 3B.} cycle shape $3^8$, order 3, level 9. The input Jacobi forms are \begin{align*} \phi_1 &:= \phi_{3B}(\tau, z) \\ &= 2 \zeta^{\pm 1} - 4 + (-4 \zeta^{\pm 2} + 4 \zeta^{\pm 1}) q + (2 \zeta^{\pm 3} - 4) q^2 + O(q^3) \end{align*} and $\phi_3 = \phi_9 = \phi_{1A}$.

Principal part: \begin{align*} &4 \mathfrak{e}_{(0,0,0)} + 2 q^{-1/4} \mathfrak{e}_{(0,1/2,0)} + 2 q^{-1/4} \mathfrak{e}_{(1/3, 1/2, 2/3)} + 2 q^{-1/4} \mathfrak{e}_{(2/3,1/2,1/3)} \\ &\quad - 6 \sum_{a \in (\mathbb{Z}/9\mathbb{Z})^{\times}} q^{-1/36} \mathfrak{e}_{(a/9, 1/2, -2a^{-1}/9) }. \end{align*}

\bigskip

\noindent \textbf{Class 4A.} cycle shape $2^4 4^4$, order 4, level 8. The input Jacobi forms are \begin{align*} \phi_1 &:= \phi_{4A}(\tau, z) \\ &= 2 \zeta^{\pm 1} - 4 + (-4 \zeta^{\pm 2} + 16 \zeta^{\pm 1} - 24)q  + (2 \zeta^{\pm 3} - 24 \zeta^{\pm 2} + 62 \zeta^{\pm 1} - 80) q^2 + O(q^3); \end{align*} $\phi_2 = \phi_{2A}$ and $\phi_4 = \phi_8 = \phi_{1A}$.

Principal part: $$4 \mathfrak{e}_{(0,0,0)} + 2 q^{-1/4} \mathfrak{e}_{(0,1/2,0)} + 2 q^{-1/4} \mathfrak{e}_{(1/2,1/2,1/2)} - 2 \sum_{a\in (\mathbb{Z}/8\mathbb{Z})^{\times}} q^{-1/8} \mathfrak{e}_{(a/8, 1/2, -a^{-1}/8)}.$$

\bigskip

\noindent \textbf{Class 4B.} cycle shape $1^4 2^2 4^4$, order 4, level 4. The input Jacobi forms are $$ \phi_1 := \phi_{4B}(\tau, z) = 2 \zeta^{\pm 1} + (2 \zeta^{\pm 3} - 2 \zeta^{\pm 1}) q^2 + O(q^3);$$  $\phi_2 = \phi_{2A}$ and $\phi_4 = \phi_{1A}$. 

Principal part: $$6 \mathfrak{e}_{(0,0,0)} + 2 q^{-1/4} \mathfrak{e}_{(0,1/2,0)}.$$

\bigskip

\noindent \textbf{Class 4C.} cycle shape $4^6$, order 4, level 16. The input Jacobi forms are \begin{align*} \phi_1 &:= \phi_{4C}(\tau, z) \\ &= 2 \zeta^{\pm 1} - 4 + (-4 \zeta^{\pm 2} + 8 \zeta^{\pm 1} - 8) q + (2 \zeta^{\pm 3} - 8 \zeta^{\pm 2} + 14 \zeta^{\pm 1} - 16) q^2 + O(q^3); \end{align*} $\phi_2 = \phi_{2B}$; $\phi_4 = \phi_8 = \phi_{16} = \phi_{1A}$.

Principal part:  \begin{align*} &2 \mathfrak{e}_{(0,0,0)} + 2q^{-1/4}e_{(0, 1/2, 0)} + 2q^{-1/4}e_{(1/4, 1/2, 3/4)} + 2q^{-1/4}e_{(1/2, 1/2, 1/2)} + 2q^{-1/4}e_{(3/4, 1/2, 1/4)} \\ &\quad - 4 \sum_{a \in (\mathbb{Z}/16\mathbb{Z})^{\times}} q^{-1/16} \mathfrak{e}_{(a/16, 1/2, -3a^{-1}/16)}. \end{align*}

\bigskip

\noindent \textbf{Class 5A.} cycle shape $1^4 5^4$, order 5, level 5. The input Jacobi forms are \begin{align*} \phi_1 &:= \phi_{5A}(\tau, z) \\ &= 2 \zeta^{\pm 1} + (2 \zeta^{\pm 1} - 4) q + (2 \zeta^{\pm 3} - 4 \zeta^{\pm 2} + 4 \zeta^{\pm 1} - 4) q^2 + O(q^3) \end{align*} and $\phi_5 = \phi_{1A}$.

Principal part: $$4 \mathfrak{e}_{(0,0,0)} + 2q^{-1/4} \mathfrak{e}_{(0,1/2,0)} - 2 \sum_{a \in (\mathbb{Z}/5\mathbb{Z})^{\times}} q^{-1/20} \mathfrak{e}_{(a/5, 1/2, -a^{-1}/5)}.$$

\bigskip

\noindent \textbf{Class 6A.} cycle shape $1^2 2^2 3^2 6^2$, order 6, level 6. The input Jacobi forms are \begin{align*} \phi_1 &:= \phi_{6A}(\tau, z) \\ &= 2 \zeta^{\pm 1} - 2 + (-2 \zeta^{\pm 2} + 6 \zeta^{\pm 1} - 8) q + (2 \zeta^{\pm 3} - 8 \zeta^{\pm 2} + 16 \zeta^{\pm 1} - 20) q^2 + O(q^3) \end{align*} and $\phi_2 = \phi_{3A}$; $\phi_3 = \phi_{2A}$; $\phi_6 = \phi_{1A}$. 

Principal part: $$4 \mathfrak{e}_{(0,0,0)} + 2 q^{-1/4} \mathfrak{e}_{(0,1/2,0)} - 2q^{-1/12} \mathfrak{e}_{(1/6,1/2,5/6)} - 2q^{-1/12} \mathfrak{e}_{(5/6, 1/2, 1/6)}.$$

\bigskip

\noindent \textbf{Class 6B.} cycle shape $6^4$, order 6, level 36. The input Jacobi forms are \begin{align*} \phi_1 &:= \phi_{6B}(\tau, z) \\ &= 2 \zeta^{\pm 1} - 4 + (-4 \zeta^{\pm 2} + 12 \zeta^{\pm 1} - 16)q + (2 \zeta^{\pm 3} - 16 \zeta^{\pm 2} + 40 \zeta^{\pm 1} - 52) q^2 + O(q^3) \end{align*} and $\phi_2 = \phi_4 = \phi_{3B}$; $\phi_3 = \phi_9 = \phi_{2B}$; $\phi_6 = \phi_{12} = \phi_{18} = \phi_{36} = \phi_{1A}$. 

Principal part: \begin{align*} &0 \mathfrak{e}_{(0,0,0)} + 2 \sum_{a \in \mathbb{Z}/6\mathbb{Z}} q^{-1/4} \mathfrak{e}_{(a/6, 1/2, -a/6)} - 2 \sum_{a \in (\mathbb{Z}/36\mathbb{Z})^{\times}} q^{-1/9} \mathfrak{e}_{(a/36, 1/2, -5a^{-1}/36)} \\
&\quad - 2 \sum_{a \in (\mathbb{Z}/9\mathbb{Z})^{\times}} q^{-1/36} \mathfrak{e}_{(a/9, 1/2, 4a^{-1}/9)}  - 4 \sum_{a \in (\mathbb{Z}/18\mathbb{Z})^{\times}} q^{-1/36} \mathfrak{e}_{(a/18, 1/2, -a^{-1}/9)} \\ & \quad - 4 \sum_{a \in (\mathbb{Z}/18\mathbb{Z})^{\times}} q^{-1/36} \mathfrak{e}_{(-a^{-1}/9, 1/2, a/18)}. \end{align*}

\bigskip

\noindent \textbf{Classes 7A/7B.} cycle shape $1^3 7^3$, order 7, level 7. The input Jacobi forms are \begin{align*} \phi_1 &:= \phi_{7AB}(\tau, z) \\ &= 2 \zeta^{\pm 1} - 1 + (-\zeta^{\pm 2} + 5 \zeta^{\pm 1} - 8) q + (2 \zeta^{\pm 3} - 8 \zeta^{\pm 2} + 17 \zeta^{\pm 1} - 22) q^2 + O(q^3) \end{align*}  and $\phi_7 = \phi_{1A}$. 

Principal part: \begin{align*} 2 \mathfrak{e}_{(0,0,0)} + 2 q^{-1/4} \mathfrak{e}_{(0,1/2,0)} - \sum_{a \in (\mathbb{Z}/7\mathbb{Z})^{\times}} q^{-3/28} \mathfrak{e}_{(a/7, 1/2, -a^{-1}/7)}. \end{align*}

\bigskip

\noindent \textbf{Class 8A.} cycle shape $1^2 2^1 4^1 8^2$, order 8, level 8. The input Jacobi forms are \begin{align*} \phi_1 &:= \phi_{8A}(\tau, z) \\ &= 2 \zeta^{\pm 1} - 2 + (-2 \zeta^{\pm 2} + 8 \zeta^{\pm 1} - 12) q + (2 \zeta^{\pm 3} - 12 \zeta^{\pm 2} + 30 \zeta^{\pm 1} - 40) q^2 + O(q^3) \end{align*} and $\phi_2 = \phi_{4B}$; $\phi_4 = \phi_{2A}$; $\phi_8 = \phi_{1A}$. 

Principal part: $$2 \mathfrak{e}_{(0,0,0)} + 2 \mathfrak{e}_{(0,1/2,0)} - \sum_{a \in (\mathbb{Z}/8\mathbb{Z})^{\times}} q^{-1/8} \mathfrak{e}_{(a/8,1/2,-a/8)}.$$

\bigskip

\noindent \textbf{Class 10A.} cycle shape $2^2 10^2$, order 10, level 20. The input Jacobi forms are \begin{align*} \phi_1 &:= \phi_{10A}(\tau,z) \\ &= 2 \zeta^{\pm 1} - 4 + (-4 \zeta^{\pm 2} + 10 \zeta^{\pm 1} -12) q + (2 \zeta^{\pm 3} -12 \zeta^{\pm 2} + 28 \zeta^{\pm 1} - 36) q^2 + O(q^3) \end{align*} and $\phi_2 = \phi_4 = \phi_{5A}$; $\phi_5 = \phi_{2B}$; $\phi_{10} = \phi_{20} = \phi_{1A}$.  

Principal part: \begin{align*} &0 \mathfrak{e}_{(0,0,0)} + 2 \mathfrak{e}_{(0,1/2,0)} + 2 \mathfrak{e}_{(1/2,1/2,1/2)}  - 2 \sum_{a \in (\mathbb{Z}/20\mathbb{Z})^{\times}} q^{-1/10} \mathfrak{e}_{(a/20,1/2,-3a^{-1}/20) }\\ &\quad - 2 \sum_{a \in (\mathbb{Z}/10\mathbb{Z})^{\times}} q^{-1/20} \mathfrak{e}_{(a/10, 1/2, -a^{-1}/10) } - 2 \sum_{a \in (\mathbb{Z}/5\mathbb{Z})^{\times}} q^{-1/20} \mathfrak{e}_{(a/5, 1/2, a^{-1}/5) }. \end{align*}

\bigskip

\noindent \textbf{Class 11A.} cycle shape $1^2 11^2$, order 11, level 11. The input Jacobi forms are \begin{align*} \phi_1 &:= \phi_{11A}(\tau, z) \\ &= 2 \zeta^{\pm 1} - 2 + (-2\zeta^{\pm 2} + 4 \zeta^{\pm 1} - 4)q + (2 \zeta^{\pm 3} - 4 \zeta^{\pm 2} + 8 \zeta^{\pm 1} - 12) q^2 + O(q^3) \end{align*} and $\phi_{11} = \phi_{1A}$. 

Principal part: \begin{align*} &0 \mathfrak{e}_{(0,0,0)} + 2 q^{-1/4} \mathfrak{e}_{(0,1/2,0)} - 2 \sum_{a \in (\mathbb{Z}/11\mathbb{Z})^{\times}} q^{-3/44} \mathfrak{e}_{(a/11, 1/2, -2a^{-1}/11) }. \end{align*}

\bigskip

\noindent \textbf{Class 12A.} cycle shape $2^1 4^1 6^1 12^1$, order 12, level 24. The input Jacobi forms are \begin{align*} \phi_1 &:= \phi_{12A}(\tau, z) \\ &= 2 \zeta^{\pm 1} - 4 + (-4\zeta^{\pm 2} + 10 \zeta^{\pm 1} - 12) q + (2 \zeta^{\pm 3} - 12 \zeta^{\pm 2} + 32 \zeta^{\pm 1} - 44) q^2 + O(q^3) \end{align*} and $\phi_2 = \phi_{6A}$; $\phi_3 = \phi_{4A}$; $\phi_4 = \phi_8 = \phi_{3A}$; $\phi_6 = \phi_{2A}$; $\phi_{12} = \phi_{24} = \phi_{1A}$. 

Principal part: \begin{align*} &0 \mathfrak{e}_{(0,0,0)} +  2 q^{-1/4} \mathfrak{e}_{(0,1/2,0)} + 2 q^{-1/4} \mathfrak{e}_{(1/2,1/2,1/2)} \\ &\quad - \sum_{a \in (\mathbb{Z}/24\mathbb{Z})^{\times}} q^{-1/8} \mathfrak{e}_{(a/24, 1/2, -a^{-1}/8)} - \sum_{a \in (\mathbb{Z}/24\mathbb{Z})^{\times}} q^{-1/8} \mathfrak{e}_{(-a^{-1}/8, 1/2, a/24)} \\ &\quad - 2 \sum_{a \in (\mathbb{Z}/12\mathbb{Z})^{\times}} q^{-1/12} \mathfrak{e}_{(a/12, 1/2, -a^{-1}/12)} - 2 \sum_{a \in (\mathbb{Z}/24\mathbb{Z})^{\times}} q^{-1/24} \mathfrak{e}_{(a/24, 1/2, -5a^{-1}/24)}. \end{align*}

\bigskip

\noindent \textbf{Class 12B.} cycle shape $12^2$, order 12, level 144. The input Jacobi forms are \begin{align*} \phi_1 &:= \phi_{12B}(\tau, z) \\ &= 2 \zeta^{\pm 1} -4 + (-4 \zeta^{\pm 2} + 8 \zeta^{\pm 1} - 8) q + (2 \zeta^{\pm 3} - 8 \zeta^{\pm 2} + 20 \zeta^{\pm 1} - 28)q^2 + O(q^3) \end{align*} and $\phi_2 = \phi_{6B}$; $\phi_3 = \phi_9 = \phi_{4C}$;$\phi_4 = \phi_8 = \phi_{16} = \phi_{3B}$; $\phi_6 = \phi_{18} = \phi_{2B}$; $\phi_{12} = \phi_{24} = \phi_{36} = \phi_{48} = \phi_{72} = \phi_{144} = \phi_{1A}$. 

Principal part: \begin{align*} &(-2) \mathfrak{e}_{(0,0,0)} + 2q^{-1/4} \mathfrak{e}_{(0,1/2,0)} + 2q^{-1/4} \mathfrak{e}_{(1/2,1/2,1/2)} + 2 \sum_{a \in (\mathbb{Z}/3\mathbb{Z})^{\times}} q^{-1/4} \mathfrak{e}_{(a/3,1/2,-a/3)} \\ 
&\quad + 2q^{-1/4} \sum_{a \in (\mathbb{Z}/4\mathbb{Z})^{\times}} q^{-1/4} \mathfrak{e}_{(a/4, 1/2,-a/4)} + 2 \sum_{a \in (\mathbb{Z}/6\mathbb{Z})^{\times}} q^{-1/4} \mathfrak{e}_{(a/6, 1/2, -a/6) } \\
&\quad + 2 \sum_{a \in (\mathbb{Z}/12\mathbb{Z})^{\times}} q^{-1/4} \mathfrak{e}_{(a/12, 1/2, -a/12)}  - \sum_{a \in (\mathbb{Z}/72\mathbb{Z})^{\times}} q^{-1/9} \mathfrak{e}_{(a/72, 1/2, -5a^{-1}/72)} \\
&\quad - \sum_{a \in (\mathbb{Z}/72\mathbb{Z})^{\times}} q^{-1/9} \mathfrak{e}_{(a/72, 1/2, 31a^{-1}/72)}  - 2 \sum_{a \in (\mathbb{Z}/144\mathbb{Z})^{\times}} q^{-13/144} \mathfrak{e}_{(a/144, 1/2, -23a^{-1}/144)} \\ &\quad - 2 \sum_{a \in (\mathbb{Z}/48\mathbb{Z})^{\times}} q^{-1/16} \mathfrak{e}_{(a/48, 1/2, 13a^{-1}/48)}   - 2 \sum_{a \in (\mathbb{Z}/48\mathbb{Z})^{\times}} q^{-1/16} \mathfrak{e}_{(a/48, 1/2, -3a^{-1}/48) } \\ &\quad - 2 \sum_{a \in (\mathbb{Z}/9\mathbb{Z})^{\times}} q^{-1/36} \mathfrak{e}_{(a/9, 1/2, a^{-1}/9)}   - 2 \sum_{a \in (\mathbb{Z}/18\mathbb{Z})^{\times}} q^{-1/36} \mathfrak{e}_{(a/18, 1/2, -5a^{-1}/18)} \\ &\quad - 2 \sum_{\substack{a, c \in \mathbb{Z}/36\mathbb{Z} \\ \mathrm{gcd}(a, c, 36) = 1 \\ a \cdot c \in \{7, 16, 34\}}} q^{-1/36} \mathfrak{e}_{(a/36, 1/2, c/36)}  - 2 \sum_{a \in (\mathbb{Z}/144\mathbb{Z})^{\times}} q^{-1/144} \mathfrak{e}_{(a/144, 1/2, -35a^{-1}/144) }. \end{align*}

\bigskip

\noindent \textbf{Classes 14A/14B.} cycle shape $1^1 2^1 7^1 14^1$, order 14, level 14. The input Jacobi forms are \begin{align*} \phi_1 &:= \phi_{14AB}(\tau, z) \\ &= 2 \zeta^{\pm 1} - 3 + (-3\zeta^{\pm 2} + 7 \zeta^{\pm 1} - 8) q  + (2 \zeta^{\pm 3} - 8 \zeta^{\pm 2} + 19 \zeta^{\pm 1} - 26)q^2 +O(q^3) \end{align*} and $\phi_2 = \phi_{7AB}$; $\phi_7 = \phi_{2A}$; $\phi_{14} = \phi_{1A}$. 

Principal part: \begin{align*} &0 \mathfrak{e}_{(0,0,0)} + 2 q^{-1/4} \mathfrak{e}_{(0,1/2,0)}  - \sum_{a \in (\mathbb{Z}/14\mathbb{Z})^{\times}} q^{-3/28} \mathfrak{e}_{(a/14, 1/2, -a^{-1}/7) } \\
& \quad - \sum_{a \in (\mathbb{Z}/14\mathbb{Z})^{\times}} q^{-3/28} \mathfrak{e}_{(-a^{-1}/7, 1/2, a/14) }  - 2 \sum_{a \in (\mathbb{Z}/14\mathbb{Z})^{\times}} q^{-1/28} \mathfrak{e}_{(a/14, 1/2, -3a^{-1}/14)}. \end{align*}

\bigskip

\noindent \textbf{Classes 15A/15B.} cycle shape $1^1 3^1 5^1 15^1$, order 15, level 15. The input Jacobi forms are \begin{align*} \phi_1 &:= \phi_{15AB}(\tau, z) \\ &= 2\zeta^{\pm 1} - 3 + (-3\zeta^{\pm 2} + 8 \zeta^{\pm 1} - 10)q + (2 \zeta^{\pm 3} - 10 \zeta^{\pm 2} + 25 \zeta^{\pm 1} - 34)q^2 + O(q^3) \end{align*} and $\phi_3 = \phi_{5A}$; $\phi_5 = \phi_{3A}$; $\phi_{15} = \phi_{1A}$. 

Principal part: \begin{align*} &0 \mathfrak{e}_{(0,0,0)} + 2 q^{-1/4} \mathfrak{e}_{(0,1/2,0)} - \sum_{a \in (\mathbb{Z}/15 \mathbb{Z})^{\times}}  q^{-7/60} \mathfrak{e}_{(a/15, 1/2, -2a^{-1}/15)}   \\ &\quad - \sum_{a \in (\mathbb{Z}/15\mathbb{Z})^{\times}} q^{-1/20} \mathfrak{e}_{(a/15, 1/2, -a^{-1}/5)} - \sum_{a \in (\mathbb{Z}/15\mathbb{Z})^{\times}} q^{-1/20} \mathfrak{e}_{(-a^{-1}/5, 1/2, a/15)}. \end{align*}

\bigskip

\noindent \textbf{Classes 21A/21B.} cycle shape $3^1 21^1$, order 21, level 63. The input Jacobi forms are \begin{align*} \phi_1 &:= \phi_{21AB}(\tau, z) \\ &= 2 \zeta^{\pm 1} - 4 + (-4 \zeta^{\pm 2} + 11 \zeta^{\pm 1} - 14) q + (2 \zeta^{\pm 3} - 14 \zeta^{\pm 2} + 35 \zeta^{\pm 1} - 46)q^2 + O(q^3) \end{align*} and $\phi_3 = \phi_9 = \phi_{7AB}$; $\phi_7 = \phi_{3B}$; $\phi_{21} = \phi_{63} = \phi_{1A}$.

Principal part: \begin{align*} &(-2) \mathfrak{e}_{(0,0,0)} + 2 q^{-1/4} \mathfrak{e}_{(0,1/2,0)} + 2q^{-1/4} \mathfrak{e}_{(1/3,1/2,2/3)} + 2q^{-1/4} \mathfrak{e}_{(2/3,1/2,1/3)} \\ &\quad - \sum_{a \in (\mathbb{Z}/63 \mathbb{Z})^{\times}} q^{-31/252} \mathfrak{e}_{(a/63, 1/2, -8a^{-1}/63)} - \sum_{a \in (\mathbb{Z}/21\mathbb{Z})^{\times}} q^{-3/28} \mathfrak{e}_{(a/21, 1/2, -a^{-1}/21)} \\ &\quad - \sum_{a \in (\mathbb{Z}/63\mathbb{Z})^{\times}} q^{-19/252} \mathfrak{e}_{(a/63, 1/2, -11a^{-1}/63)} \\ &\quad - \sum_{a \in (\mathbb{Z}/63\mathbb{Z})^{\times}} q^{-1/36} \mathfrak{e}_{(a/63, 1/2, -7a^{-1}/9)} - \sum_{a \in (\mathbb{Z}/63\mathbb{Z})^{\times}} q^{-1/36} \mathfrak{e}_{(-7a^{-1}/9, 1/2, a/63)}. \end{align*}

\bigskip

\noindent \textbf{Classes 23A/23B.} cycle shape $1^1 23^1$, order 23, level 23. The input Jacobi forms are \begin{align*} \phi_1 &:= \phi_{23AB}(\tau, z) \\ &= 2 \zeta^{\pm 1} - 3 + (-3 \zeta^{\pm 2} + 10 \zeta^{\pm 1} - 14) q + (2 \zeta^{\pm 3} - 14 \zeta^{\pm 2} + 32 \zeta^{\pm 1} - 40) q^2 + O(q^3) \end{align*} and $\phi_{23} = \phi_{1A}$. 

Principal part: \begin{align*} &(-2) \mathfrak{e}_{(0,0,0)} + 2q^{-1/4} \mathfrak{e}_{(0,1/2,0)} \\ &\quad - \sum_{a \in (\mathbb{Z}/23\mathbb{Z})^{\times}} q^{-11/92} \mathfrak{e}_{(a/23, 1/2, -3a^{-1}/23)} - \sum_{a \in (\mathbb{Z}/23\mathbb{Z})^{\times}} q^{-7/92} \mathfrak{e}_{(a/23, 1/2, -4a^{-1}/23)}. \end{align*}

\section*{Appendix B: Data for Theorem \ref{mth2}}

In this appendix we describe the correction terms $\mathbf{B}(\hat F_g)$ appearing in Theorem \ref{mth2-equiv} by listing the family of Jacobi forms $(\phi_d)_{d | N_g} = \hat{\mathbb{J}}((\phi_d)_{d | N_g})$ and the principal parts of $\hat F_g$. These terms appear only when $g$ has cycle shape 3B or 4C. 

\bigskip

\noindent \textbf{Class 3B.} cycle shape $3^8$, order 3, level 9. The input Jacobi forms are \begin{align*} \phi_1 &:= 18 \frac{\eta(\tau)^3 \eta(9 \tau)^3}{\eta(3\tau)^2} \phi_{-2, 1}(\tau, z) \\ &= (18 \zeta^{\pm 1} - 36)q + (-36 \zeta^{\pm 2} + 90 \zeta^{\pm 1} - 108) q^2 + O(q^3) \end{align*} and $\phi_3 = 0$, $\phi_9 = 0$. 

Principal part: \begin{align*} &0 \mathfrak{e}_{(0, 0, 0)} - 2q^{-1/4} \mathfrak{e}_{(1/3, 1/2, 2/3)} - 2q^{-1/4} \mathfrak{e}_{(2/3, 1/2, 1/3)} \\ &\quad + 2q^{-1/4} \mathfrak{e}_{(1/3, 1/2, 1/3)} + 2q^{-1/4} \mathfrak{e}_{(2/3, 1/2, 2/3)} \\ &-2q^{-5/36} \sum_{a \in (\mathbb{Z}/9\mathbb{Z})^{\times}} \mathfrak{e}_{(a/9, 1/2, -a^{-1}/9)} \\ &+6q^{-1/36} \sum_{a \in (\mathbb{Z}/9\mathbb{Z})^{\times}} \mathfrak{e}_{(a/9, 1/2, -2a^{-1}/9)}. \end{align*}

\bigskip

\noindent \textbf{Class 4C.} cycle shape $4^6$, order 4, level 16. The input Jacobi forms are
\begin{align*} \phi_1 &:= 16 \frac{\eta(2\tau)^4 \eta(8\tau)^4}{\eta(4\tau)^4} \phi_{-2, 1}(\tau, z) \\ &= (16 \zeta^{\pm 1} - 32)q + (-32 \zeta^{\pm 2} + 128 \zeta^{\pm 1} - 192) q^2 + O(q^3) \end{align*} and $\phi_2 = 0$, $\phi_4 = 0$, $\phi_8 = 0$, $\phi_{16} = 0$. 

Principal part: \begin{align*} &0 \mathfrak{e}_{(0,0,0)} - 2q^{-1/4} \mathfrak{e}_{(1/4, 1/2, 3/4)} - 2q^{-1/4} \mathfrak{e}_{(3/4, 1/2, 1/4)} \\ &\quad + 2q^{-1/4} \mathfrak{e}_{(1/4, 1/2, 1/4)} + 2q^{-1/4} \mathfrak{e}_{(3/4, 1/2, 3/4)} \\ &\quad - q^{3/16} \sum_{a \in (\mathbb{Z}/16\mathbb{Z})^{\times}} \mathfrak{e}_{(a/16, 1/2, -a^{-1}/16)} \\ &\quad + 4q^{-1/16} \sum_{a \in (\mathbb{Z}/16\mathbb{Z})^{\times}} \mathfrak{e}_{(a/16, 1/2, -3a^{-1}/16)}. \end{align*}

\bibliographystyle{plainnat}
\bibliofont
\bibliography{refs}

\end{document}